\documentclass[reqno]{amsart}
\usepackage{mathrsfs,amssymb}
\usepackage{amsmath,amsfonts}
\usepackage{color}
\usepackage[inner=1.0in,outer=1.0in,bottom=1.0in, top=1.0in]{geometry}
\newtheorem{theorem}{Theorem}[section]

\newtheorem{lemma}[theorem]{Lemma}
\newtheorem{remark}[theorem]{Remark}

\numberwithin{equation}{section}
\numberwithin{equation}{section}
\begin{document}
\title{On the growth of the number of totally geodesic surfaces in some hyperbolic $3$-manifolds}
\author{Junehyuk Jung}
\address{Department of Mathematics, Texas A\&M University, College Station, TX 77840}
\email{junehyuk@math.tamu.edu}
\begin{abstract}
Let $d$ be a positive square-free integer $\equiv 3 \pmod{4}$ such that there is no invariant of the ideal class group $\mathbb{Q}\lbrack \sqrt{-d}\rbrack$ which is divisible by $4$. We prove an asymptotic formula for the number of immersed totally geodesic surfaces in $\Gamma_{-d}\backslash \mathbb{H}^3$ having area less than $X$.
\end{abstract}
\thanks{We appreciate Sang-hyun Kim for suggesting the problem and many helpful discussions. We thank Peter Sarnak, Junsoo Ha, Sug Woo Shin, Hee Oh, Jeff Hoffstein, E. Mehmet Kiral, and Naser Talebizadeh Sardari for various comments related to the main result.}
\maketitle
\section{Introduction}
Let $M$ be an $n$-dimensional hyperbolic manifold. Let $\pi\left(X\right)$ be the number of closed geodesics in $M$ that has length less than $X$. A quick application of Selberg's trace formula is
\begin{equation}\label{pgt}
\pi\left(X\right) \sim\frac{e^{(n-1)X}}{(n-1)X},
\end{equation}
which is often referred as a prime geodesic theorem \cite{MR0088511, MR0126549, MR0257933, MR2630950, MR711197}. Here we use $f(X) \sim g(X)$ to mean $\lim_{X\to\infty} f(X)/g(X) = 1$.

Now let $\xi(X)$ be the number of (immersed) totally geodesic surfaces in $M$ that has area less than $X$. The main purpose of this article is to find an asymptotic formula for $\xi(X)$ that is analogous to \eqref{pgt}. We note that $\xi(X)$ can be identically zero for certain hyperbolic $3$ manifolds $M$ (see for instance, \cite{MR1937957}). Because of its subtle nature, we examine this problem for a certain class of hyperbolic $3$-manifolds.

\begin{theorem}\label{thm}
Let $d$ be a positive square-free integer $\equiv 3 \pmod{4}$. Assume that there is no invariant of the ideal class group $\mathbb{Q}\lbrack \sqrt{-d}\rbrack$ which is divisible by $4$. Let $\Gamma_{d}$ be the Bianchi group $ PSL_2(O_{d})$, where $O_{d}$ is the ring of integers of $\mathbb{Q}[\sqrt{-d}]$. Then the number of immersed totally geodesic surfaces in $\Gamma_{d}\backslash \mathbb{H}^3$ having area less than $X$ is given by
\begin{equation}\label{blah}
\xi(X) = \frac{\tau (d)\pi}{4} \prod_{p} \left(1-\frac{\chi_{-d}(p)^2+\chi_{-d}(p)}{p^2}+\frac{1}{p^3}\right) X +o(X).
\end{equation}
Here $\tau(n)$ is the number of divisors of $n$.
\end{theorem}
\begin{remark}
Similar formula can be derived for the case $d=4$,
\[
\xi(X) = \frac{5\pi}{12} \prod_{p} \left(1-\frac{\chi_{-4}(p)^2+\chi_{-4}(p)}{p^2}+\frac{1}{p^3}\right) X +o(X)
\]
where volumes of all totally geodesic surfaces are given in \cite{MR1108918}. The difference in the leading coefficient amounts to the fact that, for $d =4$, we have $\begin{pmatrix}
                                      i & 0 \\
                                      0 & -i
                                    \end{pmatrix} \in \Gamma_d$.
\end{remark}
\begin{remark}
It is likely that \eqref{blah} holds for all positive square-free integer $d\equiv 3 \pmod{4}$.
\end{remark}
\begin{remark}
Existence of asymptotic formula $\xi(X) \sim cX$ for some constant $c>0$ in the case of Bianchi group $\Gamma_{d}$ is first mentioned in \cite{sarletter}.
\end{remark}

The proof of Theorem \ref{thm} is fairly straightforward. We first identify each totally geodesic surface with a maximal Fuchsian subgroup in $\Gamma_{d}$. We then express the Fuchsian subgroup as a $\mathbb{Z}$-order of a quaternion algebra over $\mathbb{Q}$, as in \S6.4 of \cite{MR1108918}. Applying Main Theorem 39.1.8 with Remark 39.1.13 from \cite{Voight2017}, we obtain the area of the given totally geodesic surface.

For $d$ such that there is no invariant of the ideal class group $\mathbb{Q}\lbrack \sqrt{-d}\rbrack$ which is divisible by $4$, the complete parametrization of totally geodesic surfaces is given in \cite{Vulakh1993}. After computing area of all totally geodesic surfaces, we apply the Wiener--Ikehara theorem to deduce the asymptotic formula for $\xi(X)$.

It is desirable to find an analytic-geometric proof of Theorem \ref{thm}. This will require understanding of the leading coefficient in the asymptotic expansion.

\section{Preliminaries}
For each totally geodesic surface $S$, there exists (not necessarily unique) a circle or a straight line $\mathcal{C} \subset \partial \mathbb{H}^3 = \mathbb{C}^*$ such that
\begin{align*}
F_S :&= \mathrm{Stab}\left(S, \mathrm{PSL}_2(O_d)\right)\\
&= \mathrm{Stab}\left(\mathcal{C}, \mathrm{PSL}_2(O_d)\right).
\end{align*}
Such $\mathcal{C}$ satisfies an equation (see for instance, \cite{MR1937957})
\[
a|z|^2 +2\mathrm{Re}(Bz)+ c =0,
\]
where $a,c \in \mathbb{Z}$ and $B\in O_d$.

From a quick computation, one cam verify the following:
\begin{lemma}
For $\mathcal{C}$ given by $|z|^2=D$, we have
\[
\mathrm{Stab}\left(\mathcal{C}, \mathrm{PSL}_2(\mathbb{C})\right) = \left\{\begin{pmatrix} x & Dy\\ \bar{y} & \bar{x}\end{pmatrix}\right\}.
\]
\end{lemma}

We will realize the Fuchsian group $F_S$ as units of reduced norm $1$ of an integral order in a quaternion algebra over $\mathbb{Q}$.

For a field $F$ with characteristic $\neq 2$ and $a,b \in F^{\times}$, a quaternion algebra $\left(\frac{a,b}{F}\right)$ is a $4$-dimensional $F$-vector space with basis $\{1,i,j,k\}$ such that
\[
i^2=a,~j^2=b,~ij=k=-ji.
\]
For a quaternion algebra $B$ over $\mathbb{Q}$, a $\mathbb{Z}$-order $O\subset B$ is a lattice that is also a subring of $B$. Given a $\mathbb{Z}$-order $O$ in a quaternion algebra $B$ over $\mathbb{Q}$ and a prime $p$, we use the following lemma to compute the Eichler symbol $\left(\frac{O}{p}\right)=\left(\frac{O_p}{p}\right)$.
\begin{lemma}[\cite{Voight2017}]
Let $\Delta:B \to \mathbb{Q}$ be the discriminant quadratic form given by
\[
\Delta(\alpha) = \mathrm{tr}(\alpha)^2 - 4\mathrm{nrd}(\alpha).
\]
For $\varepsilon = -1,0,1$, we have $\left(\frac{O}{p}\right)=\varepsilon$ if and only if $\left(\frac{\Delta(\alpha)}{p}\right)$ takes the values $\{0,\varepsilon\}$ for $\alpha \in O_p$.
\end{lemma}
We refer the reader to \cite{Voight2017} for details about the theory of quaternion algebras.

\section{For $d \equiv 3 \pmod{4}$}
\subsection{Classification}
We first recall a classification theorem for $\mathcal{C}$ corresponding to immersed totally geodesic surfaces. Let $B_d$ be the extended Bianchi group, i.e., the maximal extension of $\Gamma_d = PSL_2(O_d)$ in $PSL_2(\mathbb{C})$.
\begin{theorem}[Theorem 6, \cite{Vulakh1993}]
Let $d$ be a positive square-free integer $\equiv 3 \pmod{4}$. Assume that there is no invariant of the ideal class group $\mathbb{Q}\lbrack \sqrt{-d}\rbrack$ which is divisible by $4$. Then each binary hermitian form is $B_d$ equivalent to one and only one of
\[
\mathcal{C}_{m,c}:\quad d|z|^2 + 2 Re (m\sqrt{-d}\bar{z}) +dc = 0,
\]
where $0\leq m < \frac{d}{2}$.
\end{theorem}
Here $B_d$ is defined as follows. For $\alpha, \beta, \ldots \in O_d$, let $N(\alpha, \beta, \ldots)$ be the norm of the ideal in $O_d$ generated by $\alpha, \beta, \ldots$. Then $B_d$ is the image of
\[
L_d = \{\sigma = \begin{pmatrix}\alpha & \gamma\\ \beta & \delta\end{pmatrix} \in M_2 (O_d): \det(\sigma) = \varepsilon N(\alpha, \beta, \gamma, \delta), \text{ for some unit } \varepsilon \in O_d\}.
\]
\begin{lemma}[(2.4), \cite{Vulakh1993}]
When $d \equiv 3 \pmod{4}$, for each square-free $r|d$, choose a pair of integers $u,v$ such that $\frac{d}{r}u- rv = 1$, and let
\[
\sigma_r = \begin{pmatrix}\sqrt{-d} & r \\ vr & - u\sqrt{-d} \end{pmatrix} \in L_d,
\]
so that $\det(\sigma_r) = r $. Then $\{\sigma_r:r<\sqrt{d}\}$ is a complete set of coset representatives of $B_d/ PGL_2(O_d)$.
\end{lemma}

Combining above two results, we see that $\{\sigma_r^{-1} \mathcal{C}_{m,c}\}$ parameterize every immersed totally geodesic surfaces, where $0\leq m<d$. (Here we replaced  $d/2$ by $d$ to amount to the fact that $PGL_2(O_d)/PSL_2(O_d)$ has two elements, the identity and $\left(\begin{array}{cc}1 & 0 \\ 0 & -1\end{array}\right)$. Action of the latter element maps $\mathcal{C}_{m,c}$ to $\mathcal{C}_{-m,c}$)

\subsection{Volume computation through quaternion algebras}
We are going to use the following volume formula to compute the area the totally geodesic surface corresponding to $\sigma_r^{-1} \mathcal{C}_{m,c}$,.

\begin{theorem}[\cite{Voight2017}]\label{thm1}
Let $B$ be an indefinite quaternion algebra over $\mathbb{Q}$, and let $O \subset B$ be a $\mathbb{Z}$-order of reduced discriminant $D$. Let $\Gamma^1(O) \subset PSL_2(\mathbb{R})$ be the discrete group associated to the group $PO^1 =O^1/\{\pm 1\}$ of units of reduced norm $1$. Then we have
\[
vol(\Gamma^1(O)\backslash \mathbb{H}) = \frac{\pi}{3} D \prod_{p|D} \lambda(O,p) / \prod_{p|D} \left\lbrack\mathbb{Z}_p^\times : \mathrm{nrd}\left(O_p^\times\right)\right\rbrack,
\]
where
\[
\lambda(O,p) = \frac{1-p^{-2}}{1-\left(\frac{O}{p}\right)p^{-1}},
\]
and $\left(\frac{O}{p}\right)$ is the Eichler symbol. Here $\mathrm{nrd}$ is the reduced norm map.
\end{theorem}

To this end, we associate to each
\[
a|z|^2+2\mathrm{Re}(Bz) + c=0
\]
a $\mathbb{Z}$ order in an appropriate quaternion algebra.

\begin{theorem}\label{thm2}
Assume that $B=b\sqrt{-d}$ is purely imaginary, and that $(a,b,c)=1$. Let $D=b^2d -ac$ and let $\rho$ be a representation of the quaternion algebra $\left(\frac{-d,D}{\mathbb{Q}}\right)$ given by
\[
\rho(i) = \begin{pmatrix}\sqrt{-d}&0\\0&-\sqrt{-d}\end{pmatrix}, \quad \rho(j) = \begin{pmatrix}0&D\\1&0\end{pmatrix}.
\]
Let $T= \begin{pmatrix}a& B\\0& 1\end{pmatrix}$, and let $\rho' = T^{-1}\rho T$. Denoting by $d_0$ the greatest common divisor of $a$, $bd$, and $c$, the stabilizer group of $a|z|^2 + 2 Re (B\bar{z}) +c = 0$ in $\Gamma_d$ is given by $P_{\rho'}(M^1)$ where
\[
M = \mathbb{Z}\left\lbrack 1, \frac{\alpha_1( 1+i)}{2} , \frac{\beta (i+1)}{2} + \frac{  bi+j}{d_0 \alpha_1 },\frac{-bd - bi -j+ij}{2d_0} \right\rbrack,
\]
for some integers $0\leq \beta < a/d_0$, and $\alpha_1 | a/d_0$. The reduced discriminant of $M$ is given by $\frac{dD}{d_0^2}$.
\end{theorem}
\begin{proof}
We first note that $T$ maps
\[
a|z|^2 + 2 Re (B\bar{z}) +c = 0
\]
to
\[
|z|^2 = |B|^2-ac=D,
\]
so the image of $\rho'$ fixes $a|z|^2+2\mathrm{Re}(B\bar{z})+c=0$.

It follows from the definition that
\[
\rho'(i) = \begin{pmatrix}\sqrt{-d} & \frac{-2b d}{a}\\0&-\sqrt{-d} \end{pmatrix} \quad \rho'(j)=\begin{pmatrix} -b\sqrt{-d}&\frac{db^2+D}{a}\\a& b\sqrt{-d}\end{pmatrix} \quad \rho'(ij) = \begin{pmatrix} bd&-c\sqrt{-d}\\-a\sqrt{-d}& -bd \end{pmatrix},
\]
and
\begin{align*}
\rho(t+xi+yj+zij) &= \begin{pmatrix} t+ \sqrt{-d}x - b\sqrt{-d}y+bdz & -\frac{2bd}{a}x+\frac{db^2+D}{a}y-c\sqrt{-d}z\\ay-a\sqrt{-d}z&t - \sqrt{-d}x+b\sqrt{-d}y-bdz \end{pmatrix}\\
&=\begin{pmatrix} t-x+by+bdz +\frac{\sqrt{-d}+1}{2}(2x - 2by) & -\frac{2bd}{a}(x-by)-cy+cz-\frac{\sqrt{-d}+1}{2}2cz\\ay+az-\frac{\sqrt{-d}+1}{2}2az& t+ x-by-bdz- \frac{\sqrt{-d}+1}{2}(2x-2by) \end{pmatrix}.
\end{align*}
Therefore this is in $M_2(O_d)$ if and only if
\[
2az,\quad 2cz,\quad ay+az,\quad 2x-2by,\quad t-x+by+bdz,\quad t+x-by-bdz, \quad -\frac{2bd}{a}(x-by) -cy+cz \in \mathbb{Z}.
\]
Substituting by
\[
k = 2az, \quad l=ay+az, \quad m=2x-2by, \quad n = t-x+by+bdz,
\]
we see that this is equivalent to
\[
k,\quad l, \quad m,\quad n,\quad \frac{c}{a} k, \quad \frac{bd}{a} k , \quad \frac{bd}{a}m + \frac{c}{a}l \in \mathbb{Z}.
\]
Let $d_0=gcd(a,bd,c)=gcd(a,d,c)$. Then one can prove that there exists an upper-triangular $2\times 2$ matrix
\[
S= \begin{pmatrix}\alpha_1 & \beta \\ 0 &\alpha_2\end{pmatrix}
\]
with $\alpha_1 \alpha_2 = a/d_0$ such that
\[
\frac{bd}{a}m + \frac{c}{a}l \in \mathbb{Z}
\]
if and only if $\begin{pmatrix}m\\l\end{pmatrix} = S \begin{pmatrix}m_0\\l_0\end{pmatrix}$ for some $m_0,l_0 \in\mathbb{Z}$. Also, $k = \frac{a}{d_0} k_0$ for some $k_0\in \mathbb{Z}$. Combining all these, any $t+xi+yj+zij$ that is mapped under $\rho'$ to an element of $M_2(O_d)$ has the following form,
\begin{align*}
t+xi+yj+zij &= \frac{k}{2a}ij + \left(\frac{l}{a}-\frac{k}{2a}\right)j + \left(\frac{ bl}{a}-\frac{bk}{2a}+\frac{m}{2}\right)i + \frac{m}{2}+n-\frac{bdk}{2a}\\
&=n+\frac{\alpha_1( 1+i)}{2} m_0 + \left(\frac{\beta (i+1)}{2} + \frac{  bi+j}{d_0 \alpha_1 }\right)l_0 + \left(\frac{-bd - bi -j+ij}{2d_0}\right)k_0,
\end{align*}
and so we complete the proof.
\end{proof}

We need to compute the Eichler symbol and the index of the image of reduced norm map of the $\mathbb{Z}$-order $M$ to apply Theorem \ref{thm1}.

\begin{lemma}\label{lem1}
Let
\[
M = \mathbb{Z}\left\lbrack 1, \frac{\alpha_1( 1+i)}{2} , \frac{\beta (i+1)}{2} + \frac{  bi+j}{d_0 \alpha_1 },\frac{-bd - bi -j+ij}{2d_0} \right\rbrack,
\]
as in Theorem \ref{thm2}. For an odd prime $p | dD/d_0^2$,
\[
\left(\frac{M}{p}\right) = \left(\frac{-d}{p}\right)+\left(\frac{D}{p}\right).
\]
Also,
\[
 \left\lbrack\mathbb{Z}_p^\times : \mathrm{nrd}\left(M_p^\times\right)\right\rbrack
\]
is equal to $2$ if $p|\left(\frac{d}{d_0}, \frac{D}{d_0}\right)$ and $1$ otherwise.
\end{lemma}
\begin{proof}

Assume for contradiction that $p|d_0$. Then because $d$ is square-free, we have $(p, d/d_0) = 1$, so $p|\frac{D}{d_0}= \frac{b^2d}{d_0} - \frac{ac}{d_0}$. From $d_0 | ac/d_0$, we see that $p|b$, contradicting the assumption that $(a,b,c) =1$. Therefore $p \nmid d_0$.

We  prove the lemma by considering two cases, when $p\nmid \alpha_1 $, and when $p | \alpha_1$. Firstly when $p \nmid \alpha_1$, we have
\[
M_p =\mathbb{Z}_p \left\lbrack 1, i, j, ij\right\rbrack,
\]
and so the lemma follows trivially by observing that
\[
\left\lbrack\mathbb{Z}_p^\times : \mathrm{nrd}\left(M_p^\times\right)\right\rbrack
\]
is $1$ or $2$ depending on whether $\mathrm{nrd}\left(M_p^\times\right) \subseteq \left(\mathbb{Z}_p^\times\right)^2$ has a non-square or not.

For the rest of the proof, we assume that $p|\alpha_1$. Assume for contradiction that $p \nmid \frac{D}{d_0}$. Then $p|\frac{d}{d_0}$, so $p|b^2\frac{ d}{d_0} - \alpha_1 \alpha_2 c = \frac{D}{d_0}$, which contradicts the assumption. Therefore $p|\frac{D}{d_0} = b^2\frac{ d}{d_0} - \alpha_1 \alpha_2 c$, which implies that $p|b\frac{d}{d_0}$.

We claim that $p|\alpha_2$. Otherwise, from
\[
a|bd\beta +\alpha_2 c,
\]
we have $p|c$. Because $(a,b,c)=1$, this implies that $p \nmid b$, so $p| \frac{d}{d_0}$. In particular, we have $p|d$, and from $p|a$, $p|c$, $p|(a,c,d) = d_0$. This is a contradiction because $(d_0,\frac{d}{d_0})=1$, proving our claim.

We now have
\[
M_p = \mathbb{Z}_p \left\lbrack 1, \alpha_1 i, d_0 \beta i + \frac{2bi +2j}{\alpha_1},-bi-j+ ij\right\rbrack,
\]
and so for $\alpha = t+ \alpha_1 i x + \left(d_0\beta i + \frac{2bi +2j}{\alpha_1}\right)y+(-bi-j+ ij)z$,
\begin{align*}
\Delta(\alpha) &= \mathrm{tr}(\alpha)^2 - 4\mathrm{nrd}(\alpha)\\
&= -4\left(d\left(\alpha_1 x+ d_0 \beta y - bz+\frac{2by}{\alpha_1}\right)^2 - D\left(\frac{2y}{\alpha_1}- z\right)^2 -dDz^2\right)\\
&=-4d\left(\alpha_1 x+ d_0 \beta y\right)^2 - 8bd(\alpha_1 x+ d_0\beta y)\left(\frac{2y}{\alpha_1}-z\right) -4 ac\left(\frac{2y}{\alpha_1}- z\right)^2 +4dDz^2\\
&\equiv -4d\left(\alpha_1 x+ d_0 \beta y\right)^2 - \frac{16bdd_0\beta y^2}{\alpha_1} -\frac{16acy^2}{\alpha_1^2} + \frac{16acyz}{\alpha_1} \pmod{p}.
\end{align*}
Now
\[
- \frac{16bdd_0\beta y^2}{\alpha_1} -\frac{16acy^2}{\alpha_1^2} = -16y^2d_0 \alpha_2\frac{ (bd\beta+\alpha_2 c)}{a}
\]
and
\[
\frac{16acyz}{\alpha_1} = 16\alpha_2 d_0 cyz
\]
are multiples of $p$, since $p|\alpha_2$. Therefore
\[
\left(\frac{M}{p}\right)= \left(\frac{-4d}{p}\right) = \left(\frac{-d}{p}\right)+\left(\frac{D}{p}\right).
\]

Likewise, we have
\[
\mathrm{nrd}(\alpha) \equiv t^2 +d(\alpha_1 x+ d_0\beta y)^2  \pmod{p},
\]
and therefore when $\mathbb{Z}_p^\times$ has a non-square unit if and only if $p|d$.
\end{proof}
\begin{lemma}\label{lem2}
Let
\[
M = \mathbb{Z}\left\lbrack 1, \frac{\alpha_1( 1+i)}{2} , \frac{\beta (i+1)}{2} + \frac{  bi+j}{d_0 \alpha_1 },\frac{-bd - bi -j+ij}{2d_0} \right\rbrack,
\]
as in Theorem \ref{thm2}. Assume that $\alpha_1$ is odd. When $p=2 | D$,
\[
\left(\frac{M}{p}\right) = (-1)^{(d^2-1)/8},
\]
and
\[
 \left\lbrack\mathbb{Z}_2^\times : \mathrm{nrd}\left(M_2^\times\right)\right\rbrack=1.
\]
\end{lemma}
\begin{proof}
Because $\alpha_1$ is odd, we have
\[
M_2 = \mathbb{Z}_2 \left\lbrack 1, \frac{1+i}{2}, j, \frac{-j+ij}{2}\right\rbrack,
\]
and so for $\alpha = t+\frac{1+i}{2}x + jy + \frac{-j+ij}{2}z$,
\begin{align*}
\Delta(\alpha) &= -4 \left(d\left(\frac{x}{2}\right)^2- D\left(y-\frac{z}{2}\right)^2 - dD\left(\frac{z}{2}\right)^2\right)\\
&= - dx^2 +D(2y-z)^2 +dDz^2\\
&\equiv -dx^2 + D(1+d)z^2 \pmod{8}.
\end{align*}
From the assumption that $d \equiv 3 \pmod{4}$, we see that this is square if $d \equiv 7 \pmod{8}$, and non-square if $d \equiv 3 \pmod{8}$. Now because
\[
\mathrm{nrd}(-1+\frac{1+i}{2} *2) = d \equiv 3 \pmod{4},
\]
we have
\[
 \left\lbrack\mathbb{Z}_2^\times : \mathrm{nrd}\left(M_2^\times\right)\right\rbrack=1. \qedhere
\]
\end{proof}
We are now ready to compute the area of the immersed totally geodesic surface corresponding to $\sigma_r^{-1} \mathcal{C}_{m,c}$.
\begin{theorem}
For $0\leq m <d$, and any $c \in \mathbb{Z}$ such that $m^2>cd$, let $d_0 = d/(m,d)$ and $D= (m^2d-cd^2)/(m,d)^2$. Then the immersed totally geodesic surface $\mathcal{S}_{m,c,r}$ corresponding to $\sigma_r^{-1} \mathcal{C}_{m,c}$ satisfies
\[
vol \left(\mathcal{S}_{m,c,r}\right) = \frac{d}{d_0^2}\frac{\pi 2^{-\omega\left(\gcd\left(d/d_0,D\right)\right)}}{3}\prod_{p|\frac{d}{d_0}}\frac{1-p^{-2}}{1-\left(\frac{D}{p}\right)p^{-1}} D\prod_{\substack{p|D\\p\nmid d}} \left(1+\chi_{-d}(p)p^{-1}\right)
\]
where $\omega(n)$ is the number of distinct prime factors of $n$, and $\chi_{-d}$ is the quadratic Dirichlet character associated to $\mathbb{Q}[\sqrt{-d}]$.
\end{theorem}
\begin{remark}
Note that we have
\[
\chi_{-d}(p) = \left(\frac{-d}{p}\right)
\]
for an odd prime $p$, and that
\[
\chi_{-d}(2) = (-1)^{(d^2-1)/8}.
\]
\end{remark}
\begin{proof}
Let $d=rs$. From direct computation, we see that $\sigma_r^{-1} \mathcal{C}_{m,c}$ satisfies
\[
d(d+2mrv+crv^2)|z|^2 +2\mathrm{Re}\left(-\sqrt{-d}\left(mr^2v+dmu+dr+cdruv\right)\bar{z}\right)+d(cdu^2+2mru+r^2)=0
\]
where $u,v$ are chosen such that
\[
su- rv = 1.
\]
Note that we can choose $v$ to be an even integer, so that $d(d+2mv+crv^2)$ is odd, hence $\alpha_1$ is odd. This equation is equivalent to
\begin{multline*}
\frac{d}{(m,d)}(s+2mv+crv^2)|z|^2 +2\mathrm{Re}\left(-\sqrt{-d}\left(\frac{m}{(m,d)}(rv+su)+\frac{d}{(m,d)}(1+cuv)\right)\bar{z}\right)\\
+\frac{d}{(m,d)}(csu^2+2 m u+r)=0.
\end{multline*}
We claim that
\begin{equation}\label{eq1}
\gcd\left(\frac{d}{(m,d)}(s+2mv+crv^2),\frac{m}{(m,d)}(rv+su)+\frac{d}{(m,d)}(1+cuv), \frac{d}{(m,d)}(csu^2+2 m u+r) \right) = 1
\end{equation}
and that
\begin{equation}\label{eq2}
\gcd\left(\frac{d}{(m,d)}(s+2mv+crv^2),d, \frac{d}{(m,d)}(csu^2+2 m u+r) \right) = \frac{d}{(m,d)}.
\end{equation}
To prove \eqref{eq1}, we first note that
\[
d\left(\frac{m}{(m,d)}(rv+su)+\frac{d}{(m,d)}(1+cuv)\right)^2 - \frac{d^2}{(m,d)^2}(s+2mv+cv^2)(csu^2+2 m u+r) = \frac{dm^2-cd^2}{(m,d)^2},
\]
and so any prime $p$ dividing the left hand side of \eqref{eq1} must divide $\frac{dm^2-cd^2}{(m,d)^2}$. Assume for contradiction that $p|d$. Then we have $p|m$, and so $p^2\|dm^2 - cd^2$. Because $p^2 \| (m,d)^2$, this is contradiction to the assumption that $p|\frac{dm^2-cd^2}{(m,d)^2}$. Therefore $p\nmid d$, and we have $m^2  \equiv cd \pmod{p} $.

Now we have
\[
d(s+2mv+crv^2) \equiv r(s^2 + 2msv+m^2v^2) \equiv r(mv+s)^2 \pmod{p}
\]
and likewise
\[
d(csu^2+2 m u+r) \equiv s( mu+r)^2 \pmod{p}.
\]
This implies that $p|(mv+s)u-(mu+r)v = 1$, and we get contradiction. This proves \eqref{eq1}.

To prove \eqref{eq2}, we first note that the left hand side of \eqref{eq2} is equal to
\[
\frac{d}{(m,d)}\gcd\left(s+2mv+crv^2,(m,d), csu^2+2 m u+r \right).
\]
If $p|\gcd\left(s+2mv+crv^2,(m,d), csu^2+2 m u+r \right)$, then because $p|d=rs$, either $p|r$ or $p|s$ should be satisfied. If $p|r$, then $p|s+2mv+crv^2$ implies $p|s$, contradicting the assumption that $d$ is square-free. If $p|s$, then $p|csu^2+2m u+r$ implies $p|r$, again contradicting the assumption that $d$ is square-free. Therefore $\gcd\left(s+2mv+crv^2,(m,d), csu^2+2 m u+r \right)=1$, proving \eqref{eq2}.

Now the theorem follows from the definition of $d_0$, $D$, and applying Theorem \ref{thm1} to the integral order computed in Theorem \ref{thm2}. Lemma \ref{lem1} and Lemma \ref{lem2} amounts to the local computation that is required in the formula given in Theorem \ref{thm1}.
\end{proof}
\subsection{Counting the surfaces when sorted by the area}
We first need an analytic lemma for counting the number of values of an arithmetic function less than $X$, as $X \to \infty$.
\begin{lemma}\label{lemmain}
Let $a$ be an integer such that,
\[
p|a \Rightarrow p|d.
\]
For any integer $r$, we have
\[
\#\{n\equiv r \pmod{a}~:~ n\prod_{p|n} \left(1+\chi_{-d}(p)p^{-1} \right)<X\} = \frac{C}{a} X + o(X)
\]
where
\[
C=\prod_{p} (1-p^{-1}+(p+\chi_{-d}(p))^{-1}).
\]
\end{lemma}
\begin{proof}
For simplicity, let $F(n) = n\prod_{p|n} \left(1+\chi_{-d}(p)p^{-1} \right)$. Note that we have
\[
F(an+r) = (a,r)F\left(\frac{a}{(a,r)}n+\frac{r}{(a,r)}\right)
\]
from the assumption on $a$, hence we may assume without loss of generality that $(a,r)=1$.

Consider Dirichlet characters $\psi_0, \ldots, \psi_{\phi(a)-1}$ modulo $a$, where we set $\psi_0$ to be the principal character. Let
\[
D_F(s,\psi_j) =  \sum_{n=1}^\infty \psi_j (n) F(n)^{-s} = \prod_{p} (1+\psi_j(p) F(p)^{-s}+\psi_j(p^2) F(p^2)^{-s}+\psi_j(p^3) F(p^3)^{-s}\ldots).
\]
Then we have for $s=\sigma+it$,
\begin{align*}
D_F(s,\psi_j) / L(s,\psi_j) &= \prod_{p} (1+\psi_j(p) F(p)^{-s}+\psi_j(p^2) F(p^2)^{-s}+\psi_j(p^3) F(p^3)^{-s}\ldots)(1-\psi_j(p)p^{-s})\\
&= \prod_{p} (1+(1+\chi_{-d}(p)p^{-1})^{-s}\psi_j(p) p^{-s}(1-\psi_j(p)p^{-s})^{-1})(1-\psi_j(p)p^{-s})\\
&=\prod_{p} (1-\psi_j(p)p^{-s}+(1+\chi_{-d}(p)p^{-1})^{-s}\psi_j(p) p^{-s})\\
&=\prod_{p} (1+O((1+|t|)p^{-1-\sigma})),
\end{align*}
where $L(s,\psi_j)$ is the Dirichlet $L$-function associated to the character $\psi_j$. This in particular implies that all $D_F(s,\psi_j)$ are holomorphic in the region $\sigma >0$, except for $D_F(s,\psi_0)$ which has a simple pole at $s=1$, where the residue is given by
\[
\lim_{s \to 1} D_F(s,\psi_0)/\zeta(s) = \prod_{p} (1-p^{-1}+(1+\chi_{-d}(p)p^{-1})^{-1} p^{-1})\prod_{p|a}(1-p^{-1}) = \frac{\phi(a)C}{a}.
\]

Now consider a Dirichlet series that converges absolutely for $\mathrm{Re}(s)>1$,
\[
\sum_{m=1}^\infty \frac{\#\{n\equiv r \pmod{a}~:~F(n)=m\}}{m^s}=\sum_{n\equiv r \pmod{a}} F(n)^{-s} =\frac{1}{\phi(a)}\sum_j D_F(s,\psi_j) \overline{\psi_j(r)}.
\]
From the computation above,
\[
\sum_{m=1}^\infty \frac{\#\{n\equiv r \pmod{a}~:~F(n)=m\}}{m^s} - \frac{C}{a(s-1)}
\]
is holomorphic in the region $\mathrm{Re}(s)>0$, and so by Wiener--Ikehara theorem (see for instance, Corollary 8.8 \cite{MR2378655}) we have
\[
\#\{n\equiv r \pmod{a}~:~F(n)\leq x\}= \sum_{m\leq x}\#\{n\equiv r \pmod{a}~:~F(n)=m\} = \frac{C}{a}x+o(x),
\]
as $x\to \infty$.
\end{proof}
With this lemma, we can count the number of totally geodesic surfaces when ordered by the area.
\begin{theorem}
With the same $C$ defined in Lemma \ref{lemmain}, we have
\[
\#\{(m,c,r):vol \left(\mathcal{S}_{m,c,r}\right) < X\} = \frac{3C\tau(d)}{2\pi }\left(\prod_{p|d} (1+p^{-2}+p^{-3}+p^{-4}+\ldots)\right) X + o(X)
\]
as $X \to \infty$.
\end{theorem}
\begin{proof}
Let $c=d/d_0 k + \kappa$ with $\kappa=0,\ldots, d/d_0 - 1$. Then
\[
\gcd(d/d_0, D) = \gcd (d/d_0, \kappa d_0^2) = \gcd (d/d_0, \kappa)
\]
and for $p|d/d_0$,
\[
\left(\frac{D}{p}\right) = \left(\frac{-\kappa d_0^2}{p}\right) = \left(\frac{-\kappa }{p}\right).
\]
Therefore for each fixed $\kappa$, $r$, and $m$,
\begin{align*}
&\#\{k:vol(\mathcal{S}_{m,d/d_0k +\kappa ,r})<X\} \\
=&\#\left\{k:\frac{d\pi 2^{-\omega(\gcd(d/d_0, \kappa))}}{3d_0^2}\prod_{p|\frac{d}{d_0}}\frac{1-p^{-2}}{1-\left(\frac{-\kappa}{p}\right)p^{-1}}F\left(-d_0dk-d_0^2\kappa +\frac{m^2 d}{(m,d)^2}\right)<X\right\}\\
=& \frac{3Cd_0}{\pi d^2} 2^{\omega\left(\gcd\left(d/d_0,\kappa\right)\right)} \prod_{p|\frac{d}{d_0}}\frac{1-\left(\frac{-\kappa}{p}\right)p^{-1}}{1-p^{-2}} X + o(X),
\end{align*}
by Lemma \ref{lemmain}.

We first sum the part depending on $\kappa$ over $0 \leq \kappa < d/d_0 $ as follows:
\begin{align*}
\sum_\kappa 2^{\omega\left(\gcd\left(d/d_0,\kappa\right)\right)} \prod_{p|\frac{d}{d_0}}\left(1-\left(\frac{-\kappa}{p}\right)p^{-1}\right) &= \sum_\kappa \prod_{p|\frac{d}{d_0}}\left(1-\left(\frac{-\kappa}{p}\right)p^{-1}\right)\left(2-\left(\frac{-\kappa}{p}\right)^2\right)\\
&=\sum_\kappa \prod_{p|\frac{d}{d_0}}\left(2-\left(\frac{-\kappa}{p}\right)^2-\left(\frac{-\kappa}{p}\right)p^{-1} \right)\\
&=\sum_\kappa \prod_{p|\frac{d}{d_0}}\left(2-\left(\frac{-\kappa}{p}\right)^2 \right)\\
&=\sum_\kappa \prod_{p|\frac{d}{d_0}}\left(2-\left(\frac{-\kappa}{p}\right)^2 \right)\\
&=\sum_\kappa 2^{\omega\left(\gcd\left(d/d_0,\kappa\right)\right)}\\
&=\sum_{e|\frac{d}{d_0}}\sum_{\substack{\kappa\\ \gcd(d/d_0,\kappa)=e}} \tau(e)\\
&=\sum_{e|\frac{d}{d_0}}\phi\left(\frac{d/d_0}{e}\right) \tau(e)\\
&=\prod_{p|\frac{d}{d_0}} \left(\phi(p)+\tau(p)\right) \\
&= \prod_{p|\frac{d}{d_0}}(p+1) = \frac{d}{d_0}\prod_{p|\frac{d}{d_0}}(1+p^{-1}).
\end{align*}
Therefore for each fixed $m$ and $r$, we have
\begin{align*}
\#\{c:vol \left(\mathcal{S}_{m,c,r}\right) < X\} =\frac{3C}{\pi d}  \prod_{p|\frac{d}{d_0}}\frac{1}{1-p^{-1}} X + o(X)
\end{align*}
Now we sum the term depending on $m$ over $0\leq m < d$ as follows:
\begin{align*}
\sum_{m}\prod_{p|\gcd(m,d)}\frac{1}{1-p^{-1}} &= \sum_{e|d} \sum_{\substack{m\\ \gcd(m,d)= e}}\prod_{p|e}\frac{1}{1-p^{-1}}\\
&=\sum_{e|d} \phi(d/e)\prod_{p|e}\frac{1}{1-p^{-1}} \\
&=\prod_{p|d} (\phi(p) + \frac{1}{1-p^{-1}})\\
&=d\prod_{p|d} (1-p^{-1}+(p-1)^{-1}),
\end{align*}
Therefore
\[
\#\{m,c:vol \left(\mathcal{S}_{m,c,r}\right) < X\} = \frac{3C}{\pi }\left(\prod_{p|d} (1+p^{-2}+p^{-3}+p^{-4}+ \ldots)\right) X + o(X),
\]
and the theorem follows by noting that the number of choices for $r$ is $\tau(d)/2$.
\end{proof}

To complete the proof of Theorem \ref{thm}, we compute the leading constant as follows:
\begin{align*}
\frac{3C\tau (d)}{2\pi} \prod_{p|d} (1+p^{-2}+p^{-3}+ \ldots)&=\frac{3C\tau (d)}{2\pi} \prod_{p|d} \frac{1+\frac{1}{p^3}}{1-\frac{1}{p^2}}\\
&=\frac{3\tau (d)}{2\pi} \prod_{p|d} \frac{1+\frac{1}{p^3}}{1-\frac{1}{p^2}} \prod_{p\nmid d} (1-p^{-1}+(p+\chi_{-d}(p))^{-1})\\
&=\frac{3\tau (d)}{2\pi} \prod_{p|d} \frac{1+\frac{1}{p^3}}{1-\frac{1}{p^2}} \prod_{p\nmid d} \frac{1+\frac{\chi_{-d}(p)}{p}-\frac{\chi_{-d}(p)}{p^2}}{1+\frac{\chi_{-d}(p)}{p}}\\
&=\frac{3\tau (d)}{2\pi} \prod_{p|d} \frac{1+\frac{1}{p^3}}{1-\frac{1}{p^2}} \prod_{p\nmid d} \frac{1-\frac{\chi_{-d}(p)^2+\chi_{-d}(p)}{p^2}+\frac{1}{p^3}}{1-\frac{1}{p^2}}\\
&=\frac{\tau (d)\pi}{4} \prod_{p} \left(1-\frac{\chi_{-d}(p)^2+\chi_{-d}(p)}{p^2}+\frac{1}{p^3}\right).
\end{align*}

\bibliographystyle{alpha}
\bibliography{bibfile}
\end{document}